\documentclass[12pt,a4paper]{amsart} %jesli chce numeracje jednostronnq, to daje ,oneside, w graniaste.
\usepackage{amsmath}
\usepackage{amsthm}
\usepackage{amssymb}
\usepackage{ot1patch}
\usepackage{mathrsfs}

\newtheorem{thm}{Theorem}[section]
\newtheorem{lem}[thm]{Lemma}

\theoremstyle{remark}

\newtheorem*{rem*}{Remark}

\theoremstyle{definition}

\newtheorem{ex}[thm]{Example}

\numberwithin{equation}{section}

\newcommand{\C}{\mathbb{C}}%{\bf C}
\begin{document}

\title{A c-holomorphic effective Nullstellensatz with parameter}

\author{Maciej P. Denkowski
}\address{Jagiellonian University, Faculty of Mathematics and Computer Science, Institute of Mathematics, \L ojasiewicza 6, 30-348 Krak\'ow, Poland}\email{maciej.denkowski@uj.edu.pl}\date{June 3rd 2014}

\begin{abstract} 
We prove a local Nullstellensatz with parameter for a continuous family of c-holomorphic functions with an effective exponent independent of the parameter: the local degree of the cycle of zeroes of the central section section. We assume that this central section defines a proper intersection and we show that we can omit this assumption in case of isolated zeroes.
\end{abstract}

\maketitle

%\bigskip
\section{Introduction}

The idea of writing this note comes from an observation made in our recent paper on the \L ojasiewicz inequality with parameter \cite{D4}. Using the methods of \cite{PT} that led to the c-holomorphic effective Nullstellensatz presented in \cite{D2} and some intersection theory results introduced in \cite{T}, we obtain an effective Nullstellensatz for a continuous family of c-holomorphic functions. 

We shall briefly recall the notion of \textit{c-holomorphic functions} introduced by Remmert in \cite{R} (see also \cite{Wh}). These are complex continuous functions defined on an analytic set (or more generally, analytic space) $A$ that are holomorphic at its regular points $\mathrm{Reg} A$. We denote by $\mathcal{O}_c(A)$ their ring for a fixed $A$. They have similar properties to those of holomorphic functions. Nevertheless, they form a larger class and do not allow the use of methods based on differentiability. Their main feature is the fact that they are characterized among all the continuous functions $A\to {\C}$ by the analycity of their graphs (see \cite{Wh}) which makes geometric methods appliable. Some effective results obtained for this class of functions from the geometric point of view are presented in \cite{D1}--\cite{D4}.  In particular, we have an identity principle on irreducible sets and a Nullstellensatz --- see \cite{D2}.

 Throughout the paper we are working with a topological, locally compact space $T$ that in addition is \textit{1st countable}. We fix also a pure $k$-dimensional analytic subset $A$ of an open set $\Omega\subset{\C}^m$ with $0\in A$ and consider a continuous function 
$f=f(t,x)\colon T\times A\to{\C}^n$. We will use the notation $f_t(x)=f(t,x)$ and $f_t=(f_{t,1},\dots, f_{t,n})$. We assume that each $f_t$ is c-holomorphic. Of course, this happens to be true iff all the components $f_{t,j}\in\mathcal{O}_c(A)$. 

Let $I_t(U)\subset\mathcal{O}_c(A\cap U)$ denote the ideal generated by $f_{t,1},\dots, f_{t_n}$ restricted to $U\cap A\neq\varnothing$ where $U$ is an open set. %Moreover, let $Z_t:=\{g\colon \mathcal{O}_c(A)\mid g^{-1}(0)$

If a mapping $h\in\mathcal{O}_c(A,{\C}^n)$ defines a proper intersection, i.e. $h^{-1}(0)$ (which corresponds to the interesection of the graph $\Gamma_h$ with $\Omega\times\{0\}^n$) has pure dimension $k-n$ (\footnote{By the identity principle presented in \cite{D2}, neither of the components of $h$ can vanish identically on any irreducible component of $A$.}), then we introduce the \textit{cycle of zeroes} for $h$ as the Draper proper intersection  cycle (\cite{Dr})
$$
Z_h:=\Gamma_h\cdot (\Omega\times\{0\}^n).
$$
In other words, $Z_h$ is a formal sum $\sum \alpha_j S_j$ where $\{S_j\}$ is the locally finite family of the irreducible components of $h^{-1}(0)$ and $\alpha_j=i(\Gamma_h\cdot(\Omega\times\{0\});S_j)$ denotes the Draper intersection multiplicity along $S_j$ (cf. \cite{Dr} and \cite{Ch}). We define the local degree (or Lelong number) of $Z_h$ at a point $a$ usually as $\mathrm{deg}_a Z_h:=\sum\alpha_j \deg_a S_j$ with the convention that $\deg_a S_j=0$ if $a\notin S_j$. 

Note that $\Gamma_h$ is a pure $k$-dimensional analytic set and $k-n$ is the minimal possible dimension for the interesection $\Gamma_h\cap(\Omega\times\{0\}^n)$ which means this interesection is what we call a proper one. 

Now we are ready to state our main result:

\begin{thm}\label{main}
Let $T,A,f$ be as above. Assume moreover that $f_{t}(0)=0$ for any $t$, and that $f_{t_0}^{-1}(0)$ has pure dimension $k-n$. Then there is a neighbourhood $T_0\times U$ of $(t_0,0)\in T\times A$ such that \begin{enumerate}
\item for all $t\in T_0$, the sets $f_t^{-1}(0)$ have pure dimension $k-n$, too;
\item for all $t\in T_0$, for any $g\in \mathcal{O}_c(A\cap U)$ vanishing on $f_t^{-1}(0)\cap U$, $g^\delta\in I_t(U)$ where $\delta=\mathrm{deg}_0 Z_{f_{t_0}}$ is independent of $t\in T_0$;
\item if $g\colon T_0\times (A\cap U)\to {\C}$ is continuous and such that each $g_t\in\mathcal{O}_c(A\cap U)$ vanishes on $f_t^{-1}(0)\cap U$, then there is a continuous function $h\colon T_0\times (A\cap U)\to {\C}^n$ with c-holomorphic $t$-sections and such that $g^\delta=\sum_{j=1}^n h_{j}f_{j}$ where $\delta$ is as above.
\end{enumerate} 
\end{thm}

This is somehow related to some results of \cite{GGVL}. In the particular case when $\dim f_{t_0}^{-1}(0)=0$, we relax the assumptions in the last Section.

\section{Proof of the main result}

The proof of Theorem \ref{main} will be derived in several steps from some results presented in \cite{D4}.

First, we easily observe that \textsl{all} the results from \cite{D4} Section 2 hold true (with exactly the same proofs) for the mapping $f\colon T\times A\to{\C}^n$ provided $f_{t_0}^{-1}(0)$ has pure dimension $k-n$, i.e. the intersection $\Gamma_{f_{t_0}}\cap(\Omega\times\{0\}^n)$ is proper.

We will briefly state in clear the main arguments. But first let us recall that a family $Z_t=\sum_s \alpha_{t,s} S_{t,s}$ ($t\in T$) of positive (\footnote{i.e. with non-negative integer coefficients $\alpha_{t,s}$}) analytic $k$-cycles (\footnote{i.e. each irreducible set $S_{t,s}$ is pure $k$-dimensional and their family is locally finite for $t$ fixed.}) in $\Omega$ converges to a positive $k$-cycle $Z_{t_0}$ when $t\to t_0$ in the sense of Tworzewski \cite{T} (see the Introduction in \cite{D4}) if\begin{itemize}
\item the analytic sets $|Z_t|:=\bigcup_s S_{t,s}$ converge to $|Z_{t_0}|$ in the sense of Kuratowski, i.e. on the one hand, for any $x\in |Z_{t_0}|$ and any sequence $t_\nu\to t_0$ there are points $|Z_{t_\nu}|\ni x_\nu\to x$, while on the other any limit point $x$ of a sequence $x_\nu\in |Z_{t_\nu}|$ chosen for $t_\nu\to t_0$, belongs to $|Z_{t_0}|$ (cf. \cite{D4} Lemma 1.4);
\item the multiplicities are presevered: for regular point $a\in\mathrm{Reg} |Z_{t_0}|$ and any relatively compact submanifold $M\Subset \Omega$ of codimension $k$ transversal to $|Z_{t_0}|$ at $a$ and such that $\overline{M}\cap |Z_{t_0}|=\{a\}$, there is $\deg(Z_t\cdot M)=\deg (Z_{t_0}\cdot M)$ for all $t$ in a neighbourhood of $t_0$, where $\deg$ denotes here the total degree of the intersection cycle (\footnote{By \cite{TW} the intersections $|Z_t|\cap M$ are finite and proper; the \textit{total degree} is the sum of the intersection multiplicites computed at the intersection points.}).
\end{itemize}
We write then $Z_t\stackrel{T}{\longrightarrow} Z_{t_0}$.

\begin{proof}[Proof of Theorem \ref{main}] 
Since $f$ is continuous, we obtain as in \cite{D4} Lemma 2.5, the Kuratowski convergence of the graphs $\Gamma_{f_t}$ to $\Gamma_{f_{t_0}}$ which together with the remark that this in fact is the local uniform convergence --- which on $\mathrm{Reg} A$ implies also the convergence of the differentials --- allows us to conclude, as in Proposition 2.6 from \cite{D4}, that the graphs converge in the sense of Kuratowski. 

Now, since $f_{t_0}^{-1}(0)$ has pure dimension $k-n$, then \cite{D4} Proposition 1.7 based on the main result of \cite{TW} shows that $f_t^{-1}(0)$ have pure dimension $k-n$ for all $t$ sufficiently close to $t_0$, which gives (1). Moreover, these zero-sets converge to $f_{t_0}^{-1}(0)$ in the sense of Kuratowski, and so applying Lemma 3.5 from \cite{T} we obtain the Tworzewski convergence $Z_{f_{t}}\stackrel{T}{\longrightarrow} Z_{f_{t_0}}$, just as in Theorem 2.7 from \cite{D4}.

We can choose coordinates in ${\C}^m$ in such a way that $f_{t_0}^{-1}(0)$ projects properly onto the first $k-n$ coordinates and we have for the intersection multiplicity the following equality: 
$$
i((\{0\}^{k-n}\times{\C}^{m-k+n})\cdot Z_{f_{t_0}};0)=\mathrm{deg}_0 Z_{f_{t_0}}.
$$
Define $\ell\colon {\C}^m\to {\C}^{k-n}$ as the linear epimorphism for which $\mathrm{Ker}\ell=\{0\}^{k-n}\times{\C}^{m-k+n}$ and take $$\varphi_t\colon A\ni x\mapsto (f_t(x),\ell(x))\in{\C}^n\times{\C}^{k-n}$$ for $t\in T$. We can find a polydisc $V\times W\subset{\C}^{k-n}\times{\C}^{m+k-n}$ centred at zero such that $$(\{0\}^{k-n}\times\overline{W})\cap f_{t_0}^{-1}(0)=\{0\}$$
and $f_{t_0}^{-1}(0)$ projects properly onto $V$.

Clearly,  $(\{0\}^{k-n}\times\overline{W})\cap f_{t_0}^{-1}(0)$ seen in  $V\times W\times \{0\}^k$ is exactly   $$(\overline{V\times W}\times \{0\}^n)\cap\Gamma_{\varphi_{t_0}}$$ where we note that $\Gamma_{\varphi_{t_0}}$ is a pure $k$-dimensional analytic set. Thus, there is a polydisc $P\subset{\C}^k$ such that $(V\times W\times P)\cap \Gamma_{\varphi_{t_0}}$ projects properly \textit{onto} $P$ along $V\times W$. This means that $\varphi_{t_0}|_{(V\times W)\cap A}$ is a proper mapping and its image is $P$.

Obviously, the mapping $$\Phi\colon T\times A\ni (t,x)\mapsto \varphi_t(x)\in{\C}^k$$ is continuous which implies (again as in \cite{D4} Proposition 2.6) that the graphs $\Gamma_{\varphi_t}$ converge in the sense of Kuratowski to $\Gamma_{\varphi_{t_0}}$ as $t\to t_0$. But the type of convergence implies that for all $t$ sufficiently close to $t_0$, the natural projection
$(V\times W\times P)\cap \Gamma_{\varphi_{t}}\to P$  
is a branched covering over $P$. In particular, all these $\varphi_t$ have the same image $P$. Let $q_t$ denote the multiplicity of the branched covering $\varphi_t|_{A\cap (V\times W)}$. 

In order to give a precise value of $q_{t}$ we remark first that actually it is the multiplicity of the projection
$$
\pi\colon {\C}^{k-1}\times{\C}^{m-k+1}\times{\C}\ni(u,v,w)\mapsto (w,u)\in{\C}\times {\C}^{k-1}
$$
over $P$ when restricted to $\Gamma_t:=\Gamma_{f_t}\cap (V\times W\times {\C})$. The classical Stoll Formula says that this multiplicity is the sum of the local multiplicities of the projection at the points of any fibre. On the other hand, as already noted in \cite{Dr}, this local multiplicity of the projection is nothing else but the isolated proper intersection multiplicity of the set we are projecting and the fibre of the linear projection. Hence, $q_t$ is just the total degree of the intersection cycle $\pi^{-1}(0)\cdot \Gamma_t$, i.e. $q_t=\deg((V\times W\times\{0\}^k)\cdot \Gamma_{t})$ (cf. the Stoll Formula). 
Moreover, in view of \cite{TW2} Theorem 2.2, we can write
\begin{align*}
(V\times W\times\{0\}^k)\cdot \Gamma_{t}&=(\{0\}^{k-n}\times W)\cdot_{V\times W\times\{0\}^k}((V\times W\times\{0\}^k)\cdot \Gamma_t)=\\
&=(\{0\}^{k-n}\times W)\cdot Z_{f_t|_{A\cap (V\times W)}}=\\
&=(\{0\}^{k-n}\times W)\cdot Z_{f_t},
\end{align*}
and so 
$$
q_t=\deg((\{0\}^{k-n}\times W)\cdot Z_{f_t}).
$$
Besides, thanks to the convergence $Z_{f_{t}}\stackrel{T}{\longrightarrow} Z_{f_{t_0}}$, we have for all $t$ in a neighbourhood $T_0$ of $t_0$
$$
\deg((\{0\}^{k-n}\times W)\cdot Z_{f_t})=\deg((\{0\}^{k-n}\times W)\cdot Z_{f_{t_0}}).
$$
From this, eventually, we obtain $q_t=\deg_0 Z_{f_{t_0}}$, $t\in T_0$.

Once we have established that, we may directly use Lemma 3.1 from \cite{D2} (which is the c-holomorphic counterpart of Lemma 1.1 in \cite{PT}) getting precisely the statement (2) for $T_0$ and $U:=V\times W$. 

In order to obtain (3) it is enough to take a closer look at how \cite{D2} Lemma 3.1 is proved. Write $A_0=A\cap U$ and consider 
$\varphi:=\Phi|_{T_0\times A_0}$. 
This is a continuous mapping with c-holomorphic sections $\varphi_t=(\varphi_{t,1},\dots, \varphi_{t_k})$ that are proper mappings $A\to P\subset {\C}^k$ of multiplicity $\delta=\deg_0 Z_{f_{t_0}}$. Take a continuous $g(t,x)$ with c-holomorphic $t$-sections vanishing on $\{\varphi_{t_1}=\ldots=\varphi_{t,n}=0\}$, i.e. on $f_t^{-1}(0)\cap U$. The properness of $\varphi_t$ allows us to define for each $t_\in T_0$ the characteristic polynomial of $g_t$ with respect to $\varphi_t$, setting
$$
p_t(w,s)=\prod_{j=1}^\delta (s-g_t(x^{(j)}))=s^\delta+\sum_{j=1}^\delta a_j(t,w)t^{\delta-j}, \quad (w,s)\in P\times{\C},
$$
where $\varphi_t^{-1}(w)=\{x^{(1)},\dots, x^{(\delta)}\}$ consists of exactly $\delta$ points and $a_j(t,w)=(-1)^j\sum_{1\leq i_1<\ldots<i_j\leq\delta}\prod_{r=1}^j g_t(x^{(r)})$ are continuous, and extend the coefficients $a_j(t,\cdot)$ through the critical locus of the branched covering $\varphi_t$ thanks to the Riemann Theorem. We obtain thus $p_t\in\mathcal{O}(P)[s]$ and, clearly, the mapping $$p\colon T_0\times A_0\times {\C}\ni (t,w,s)\mapsto p_t(w,s)\in{\C}$$ is continuous.

Observe that 
$$
p^{-1}(0)=\{(t,\varphi_t(x),g_t(x))\in T_0\times P\times {\C}\mid x\in A_0\}.
$$
The set $P\subset{\C}^k$ is a polydisc centred at zero, hence we can write $P=P_1\times\ldots\times P_k=(P_1\times\ldots\times P_n)\times P''$. Now, by the assumptions,
$$p^{-1}(0)\cap (T_0\times \{0\}^n\times P''\times{\C})=T_0\times \{0\}^n\times P''\times\{0\},
$$
which implies that $a_j|_{T_0\times \{0\}^n\times P''}\equiv 0$. We have the following parameter version of the Hadamard Lemma:
\begin{lem}
Let $h=h(t,x,y)\colon T\times D\times U\to {\C}$ be a continuous function where $T$ is a locally compact, 1st countable topological space, $D\subset{\C}^n$ is a convex neighbourhood of the origin and $U\subset{\C}^r$ is an open set. If $h_t\in\mathcal{O}(D\times U)$ and $h|_{T\times\{0\}^n\times U}\equiv 0$, then 
$$
h(t,x,y)=\sum_{j=1}^n x_jh_j(t,x,y),\quad (t,x,y)\in T\times D\times U
$$
for some continuous functions $h_j$, holomorphic in $(x,y)$. 
\end{lem}
\begin{proof}
Since $h$ is continuous, then for each $t_0$ and $t_\nu\to t_0$, we see as in \cite{D4} Lemma 2.5 that $h_{t_\nu}$ converge to $h_{t_0}$ locally uniformly. But htese are holomorphic functions, hence $\frac{\partial h_{t_\nu}}{\partial x_j}$ converge locally uniformly to $\frac{\partial h_{t_0}}{\partial x_j}$. This in turn implies that that the functions $\tilde{h}_j(t,x,y)=\frac{\partial h_{t}}{\partial x_j}(x,y)$ are continuous (and holomorphic in $(x,y)$). 

Then we proceed as Hadamard: thanks to the convexity of $D$ we can write 
$$
h(t,x,y)=\int_{[0,1]}\sum_{j=1}^n x_j\tilde{h}_j(t,sx,y)\ ds,\quad (t,x,y)\in T\times D\times U.
$$
Therefore, $h_j(t,x,y):=\int_{[0,1]}\tilde{h}_j(t,sx,y)\ ds$ are the functions looked for. Their continuity follows from classical analysis.
\end{proof}

Applying this Lemma to $a_j$ with $w=(w_1,\dots, w_n, w'')\in P_1\times\ldots\times P_n\times P''$ we obtain $a_j(t,w)=\sum_{i=1}^n w_i a_{ji}(t,w)$ with $a_{ji}$ continuous functions with holomorphic $t$-sections. 

Finally, (3) follows from the identity $p(t,\varphi(t,x),g(t,x))\equiv 0$. This ends the proof of Theorem \ref{main}.
\end{proof}

\section{A particular case: isolated zeroes.}

Basing on \cite{D2} Theorem 4.1 (cf. \cite{S} and \cite{D1}) we can obtain a counterpart of the main theorem in the case of improper isolated intersection (a c-holomorphic parameter version of the main result of \cite{Cg}).

Observe that for a c-holomorphic mapping $h\colon A\to{\C}^n$ on a pure $k$-dimensional analytic subset of some open set in ${\C}^m$ containing zero and such that $h^{-1}(0)=\{0\}^m$ we have $$\mathrm{deg}_0 Z_h=i(\Gamma_h\cdot ({\C}^m \times\{0\}^n);0)$$
where the isolated intersection multiplicity $i_0(h):=i(\Gamma_h\cdot ({\C}^m \times\{0\}^n);0)$ is computed according to Draper \cite{Dr} (when $n=k$ i.e. the proper interesection case), or according to the extension of  Achilles-Tworzewski-Winiarski \cite{ATW} in the improper case (i.e. $n>k$; note that always $n\geq k$, for $h$ is necessarily proper in a neighbourhood of zero). In the proper intersection case, this multiplicity is just the geometric multiplicity (covering number) $m_0(f)$.

\begin{thm}
Let $T,A,f$ be as earlier. Assume moreover that $f_{t}(0)=0$ for any $t$, and that $f_{t_0}^{-1}(0)=\{0\}^m$. Then there is a neighbourhood $T_0\times U$ of $(t_0,0)\in T\times A$ and for each $t\in T_0$ there is a neighbourhood $V_t\subset U$ such that for $\delta=i_0(f_{t_0})$, \begin{enumerate}
\item for all $t\in T_0$, $\#\overline{U}\cap f_t^{-1}(0)=\# {U}\cap f_t^{-1}(0)\leq \delta$, $\overline{U}\cap f_{t_0}^{-1}(0)=\{0\}^m$;
\item for all $t\in T_0$, $\overline{V_t}\cap f_t^{-1}(0)=\{0\}$ and for any $g\in \mathcal{O}_c(A\cap V_t)$ such that $g(0)=0$,
%$g$ vanishes on $U\cap f_t^{-1}(0)$, 
$g^\delta\in I_t(V_t)$ where $\delta$ is independent of $t\in T_0$;
\item if $f^{-1}(0)\cap (T_0\times U)=T_0\times\{0\}^m$ and if $g\colon T_0\times (A\cap U)\to {\C}$ is continuous and such that each $g_t\in\mathcal{O}_c(A\cap U)$ vanishes at the origin, then there is a continuous function $h\colon T_0\times (A\cap U)\to {\C}^n$ with c-holomorphic $t$-sections and such that $g^\delta=\sum_{j=1}^n h_{j}f_{j}$.
\end{enumerate} 
\end{thm}
\begin{proof}
Fix a neighbourhood $V\subset{\C}^m$ of zero such that $f_{t_0}$ is a proper mapping on $V\cap A$ over some neighbourhood $G\subset{\C}^n$ of the origin. then $X_0:=f_{t_0}(V\cap A)$ is an analytic, pure $k$-dimensional subset of $G$, by the Remmert Theorem. Consider a linear epimorphism $L\colon {\C}^n\to {\C}^k$ such that $\mathrm{Ker} L\cap C_0(X_0)=\{0\}^n$. Then $F:=(L\circ f)\colon T\times A\to {\C}^k$ satisfies the assumptions of our Main Theorem. Let $T_0, U$ and $\delta$ be as in the Main Theorem applied to $F$. Write $A_0:=A\cap U$. We restrict our considerations to this set.

Observe that $F_t=L\circ f_t$. Therefore, since $\deg_0 Z_{F_{t_0}}=i(\Gamma_{F_{t_0}}\cdot(U\times \{0\}^k);0)$ and we check as in the proof of Theorem 2.6 in \cite{D1} that the latter is equal to $i(\Gamma_{f_{t_0}}\cdot(U\times\{0\}^n);0)$, we obtain $\deg_0 Z_{F_{t_0}}=i_0(f_{t_0})$. 

Now, as we have $Z_{F_{t}}\stackrel{T}{\longrightarrow} Z_{F_{t_0}}=i_0(f_{t_0})\{0\}^m$ (by the previous proof), the type of convergence implies that --- possibly after shrinking $U$ so that the conclusions of the Main Theorem hold true in a neighbourhood of the compact set $\overline{U}$ --- we have (1). Indeed, since $\deg Z_{F_{t}}=\deg Z_{F_{t_0}}$ whenever $t$ is sufficiently close to $t_0$, we have $\#F_t^{-1}(0)=\#|Z_{F_t}|\leq \deg Z_{F_t}=\delta$. But 
$$
F_t(x)=0\ \Leftrightarrow\ x\in f_t^{-1}(\mathrm{Ker} L)
$$
and obviously, $f_t^{-1}(0)\subset f_t^{-1}(\mathrm{Ker} L)$. 

Assume for the moment that for all $t\in T_0$ we have $\mathrm{Ker} L\cap f_t(A)=\{0\}^k$ so that $f_t(x)=0$ iff $F_t(x)=0$. Hence, if $g$ vanishes on the zeroes of $f_t$, it does so on the zeroes of $F_t$. Property (2) follows now from the obvious fact that for the components of the mapping we have $(L\circ f_t)_j=L_j\circ f_t$ and as $L_j$ is a linear form, we can write it in the form $L_j(w)=\sum_{i=1}^na_{j,i}w_i$ for some $a_{j,i}\in{\C}$. Therefore, if $g_t^\delta$ is a combination of $(L\circ F)_{t,j}$ with some coefficients $h_{t,i}\in\mathcal{O}_c(A_0)$, $i=1,\dots, n$, we have \begin{align*}
g_t(x)^\delta&=\sum_{i=1}^n h_{t,i}(x) (L_j\circ f_t)(x)=\\
&=\sum_{i=1}^n h_{t,i}(x)\sum_{\kappa=1}^n a_{j,\kappa} f_{t,\kappa}(x).
\end{align*}
It remains to put $\tilde{h}_{t,j}:=\sum_{i=1}^n a_{j,i} h_{t,i}\in\mathcal{O}_c(A_0)$ to get (2).

In the general case, since $\dim f_t^{-1}(0)=0$ and $f_t(0)=0$, we obviously can find neighbourhoods $V_t$ isolating the origin in the fibre for $t\in T_0$. The previous arguments applied for a fixed $t\in T_0$ assure that for each $g\in\mathcal{O}_c(A\cap V_t)$ vanishing at zero, we have $g^{\delta_t}\in I_t(V_t)$ with $\delta_t=\mathrm{deg}_0 Z_{F_t}=i_0(f_t)$. But $\mathrm{deg}_0 Z_{F_t}\leq \deg Z_{F_t}$ and we have already shown that the latter does not exceed $\delta$, whence $\delta_t\leq\delta$. It remains to observe that $g^\delta=g^{\delta-\delta_t}\cdot g^{\delta_t}\in I_t(V_t)$ and (2) is proved.

Finally, (3) follows easily from the above and Theorem \ref{main} (3).
\end{proof}
It should be pointed out that the Kuratowski convergence of the images $X_t:=f_t(A)\stackrel{K}{\longrightarrow} f_{t_0}(A)=:X_{t_0}$ (\footnote{Which we have by Lemma 4.4 from \cite{DD}.}) together with $f_t(0)=0$ for all $t$, does not imply the Kuratowski convergence of the tangent cones. In particular if $\Lambda$ is a linear subspace transversal to $C_0(X_0)$, it need not be transversal to any other $X_t$.
\begin{ex}
Let $f(t,x)=(x^2,tx)$ for $(t,x)\in{\C}^2$ and $t_0:=0$. It satisfies the assumptions of our Theorem, but $X_t=\{t^2u=v^2\}$ are parabol{\ae} converging to the $u$-axis $X_0=\{v^2=0\}$ counted twice for the Tworzewski convergence, by the way). Clearly $\Lambda:=C_0(X_t)=\{u=0\}$ ($t\neq 0$) is transversal to $C_0(X_0)=\{v=0\}$. 
\end{ex}

Note also that $L\circ f_t$ can indeed produce some extra zeroes which explains why we were not able to find a neighbourhood of zero independent of the parameter in the proof above:
\begin{ex}
Let $A=\{xy=0\}$ be the union of the two axes in ${\C}^2$. Put for $(t,x,y)\in {\C}\times A$,
$$
f(t,x,y)=\begin{cases}
(x^2, tx), &\textrm{if}\ y=0,\\
(y^2+y, ty), &\textrm{if}\ x=0.
\end{cases}
$$
The assumptions of our Theorem are satisfied for $t_0=0$. Now, $X_0$ is simply the $x$-axis in ${\C}^2$, while $X_t$ ($t\neq 0$) consists of two parabol{\ae} of equations $t^2x=y^2$ and $t^2x=y(y+t)$. 

Now, the $y$-axis $\Lambda$ is transversal to $X_0$, but intersects $X_t$ in two points: the origin and $(0,-t)$. Hence for $L(x,y)=x$ we have $(L\circ f_t)^{-1}(0)=\{(0,0), (0,-t)\}$, whereas $f_t$ itself vanishes only at the origin. 
\end{ex}

\section{Acknowledgements}

During the preparation of this paper the author was partially supported by Polish Ministry of Science and Higher Education grant  1095/MOB/2013/0. 

This article was written during the author's stay at the University Lille 1 whose staff he thanks for excellent working conditions.

\end{document}